\documentclass[12pt]{amsart}
\usepackage{amssymb, amsthm}
\usepackage{mathtools}
\usepackage{hyperref}
\usepackage{fullpage}
\usepackage{cleveref}

\newtheorem*{theorem*}{Theorem}
\newtheorem{theorem}{\textbf{Theorem}}[section]
\newtheorem{proposition}[theorem]{\textbf{Proposition}}
\newtheorem{corollary}[theorem]{\textbf{Corollary}}
\newtheorem{lemma}[theorem]{\textbf{Lemma}}

\theoremstyle{definition}
\newtheorem{definition}[theorem]{Definition}
\newtheorem{remark}[theorem]{Remark}
\newtheorem{example}[theorem]{Example}

\newtheorem{conjecture}[theorem]{Conjecture}

\numberwithin{equation}{section}

\DeclareMathOperator{\GL}{GL}
\DeclareMathOperator{\SL}{SL}  
\DeclareMathOperator{\PGL}{PGL}     \DeclareMathOperator{\End}{End} \DeclareMathOperator{\ff}{ff} \DeclareMathOperator{\nff}{nff} \DeclareMathOperator{\cf}{cf} \DeclareMathOperator{\PP}{\mathbb{P}}

\title{Integrability of the Six-Vertex Model and the Yang-Baxter Groupoid}
\author{Slava Naprienko}

\begin{document}

\begin{abstract}
    We study the Yang-Baxter equation for the $R$-matrices of the six-vertex model. We analyze the solutions and give new parametrizations of the Yang-Baxter equation. In particular, we find the maximal commutative families of parametrized solutions which generalize the $R$-matrices from the affine quantum (super)-groups. Then we give a new parametrization of the Yang-Baxter equation by a groupoid of non-free-fermionic matrices. In the appendix, we study the general algebraic structure of the solutions of the Yang-Baxter and formulate a conjecture that extends the conjecture by Brubaker, Bump, and Friedberg that the composition law on the Yang-Baxter solutions is always associative.
\end{abstract}

\maketitle

\section{Introduction}
The six-vertex model is a famous exactly solvable model in statistical physics with a wide range of applications in various fields of mathematics and physics. It was introduced by Pauling in 1935 \cite{Pau35}. 

In \cite{Bax82}, Baxter introduced a new way of solving the six-vertex model based on a local equation for the weights that gives a sufficient condition for the row transfer matrices to commute. Independently, Yang \cite{Y67} introduced the same equation in his work on the exact solution of the related ice-type models. The equation is now called the Yang-Baxter equation (YBE) and is one of the most important equations in exactly solvable models. 

Later, the Yang-Baxter equation was reformulated in the context of the representation theory and the affine quantum (super)-groups \cite{D87}. For the six-vertex model, the relevant quantum groups are the affine quantum group $U_q(\widehat{\mathfrak{sl}}(2))$ and the affine quantum supergroup $U_q(\widehat{\mathfrak{sl}}(1|1))$. They give examples of the solutions parametrized by $\mathbb{C}$ and $\mathbb{C}^\times$ that we revisit in \Cref{example:quantumgroups}.

Another source of the parametrized solutions comes from the direct analysis of the Yang-Baxter equation. In \cite{KBI93, BS82}, an example of parametrized Yang-Baxter equation by a non-commutative group $\SL_2(\mathbb{C})$ was given. In \cite{BBF11}, authors extended the result to a larger group $\GL_2(\mathbb{C}) \times \GL_1(\mathbb{C})$. We revisit their result in \Cref{sec:pybe}.

In the present paper, we study the parametrized Yang-Baxter equation for the six-vertex model and give new parametrized solutions. In particular, we find the maximal commutative families of parametrized solutions that generalize the $R$-matrices from the affine quantum (super)-groups. They also include the combinatorial five-vertex matrices, which could be considered as the combinatorial $R$-matrices of the crystallization of the corresponding affine quantum groups. We discuss it in \Cref{example:quantumgroups} and \Cref{example:fivevertex}.

\begin{theorem*}
    For any constants $q_1,q_2 \in \mathbb{C}$, and $\beta \in \mathbb{C}$, we have the parametrized Yang-Baxter equations
    \begin{align}
        [[R^{\cf}_{q_1,q_2;\beta}(g),R^{\cf}_{q_1,q_2;\beta}(gh),R^{\cf}_{q_1,q_2;\beta}(h)]], \quad g,h \in (\mathbb{C}^\times)^3,\\
        [[R^{\ff}_{q_1,q_2;\beta}(g),R^{\ff}_{q_1,q_2;\beta}(gh),R^{\ff}_{q_1,q_2;\beta}(h)]], \quad g,h \in (\mathbb{C}^\times)^3,
    \end{align}
    where the maps $R_{q_1,q_2;\beta}^{\cf}$ and $R_{q_1,q_2;\beta}^{\ff}$ are defined by \Cref{eq:rmatrices}.
\end{theorem*}

Our main result is the parametrized Yang-Baxter equation, where the parameter space is not a group, but a \textit{groupoid} consisting of non-free-fermionic six-vertex matrices. To our knowledge, it is the first example of the parametrized Yang-Baxter equation with a non-group parameter space.

\begin{theorem*}
    We have the parametrized Yang-Baxter equation
    \[
        [[\pi(g), \pi(g*h), \pi(h)]] = 0, \quad g,h \in G_{\nff},
    \]
    where the non-free-fermionic groupoid $G_{\nff}$ is defined in \Cref{sec:groupoid}.
\end{theorem*}

When the parametrization is restricted to the vertex groups in the groupoid, we get the commutative parametrization from the previous result. However, the unique feature of the groupoid structure is the interactions between the different vertex groups. In particular, associativity is a non-trivial property as it does not come from a group structure as in previous examples. 

In Section 2 of \cite{BBF11}, authors conjectured that the composition law on the solutions of the Yang-Baxter equations is always associative. They give a heuristic argument suggesting that a composition law is always associative. This conjecture is supported by examples of the parametrized solutions which satisfy associativity. Our result provides additional evidence to the conjecture of Brubaker, Bump, and Friedberg. In fact, we prove even more: if $u*v$ and $v*w$ are defined, then so is $u*v*w$. It motivates the extended conjecture that we formulate in \Cref{conj:associativity}.

\subsection*{Acknowledgments}
I am very grateful to Daniel Bump for all his support and guidance throughout this project. Thank you!

\section{Yang-Baxter equation for the six-vertex model}
The six-vertex model in statistical mechanics can be described algebraically in terms of the matrices of weights for each vertex. See \cite{BBF11}, Section 1 for details. We study the Yang-Baxter equation for the matrices which arise from the six-vertex model.

\begin{definition}
    A \textit{six-vertex matrix} $u \in \GL_4(\mathbb{C})$ is an invertible $4 \times 4$ matrix of the form 
    \begin{equation}\label{eq:sixvertex}
         u = \begin{pmatrix}
            a_1(u) & & & \\
            & c_1(u) & b_1(u) &\\
            & b_2(u) & c_2(u) &\\
            & & & a_2(u)
        \end{pmatrix}.
    \end{equation}
    The corresponding projective matrix $[u] \in \PGL_4(\mathbb{C})$ is called a \textit{projective six-vertex matrix}. 
\end{definition}

Let $V \cong \mathbb{C}^2$ with the standard basis $e_1,e_2$. We view a six-vertex matrix as a matrix of an operator $u \in \End(V \otimes V)$ in basis $e_1 \otimes e_1, e_1 \otimes e_2, e_2 \otimes e_1, e_2 \otimes e_2$. By abuse of notation, we denote by $u$ both the six-vertex matrix and the corresponding operator. 

The determinant of a six-vertex matrix $u \in S$ is given by 
\[
    \det(u) = a_1(u)a_2(u)(c_1(u)c_2(u) - b_1(u)b_2(u)).
\]

We define the dual six-vertex matrix $u^* \in S$ of a matrix $u \in S$ as the adjugate matrix $u^* = \det(u)u^{-1}$. Explicitly, $u^*$ is given by 
\[
    u^* = \begin{pmatrix}
        a_1^*(u) & & & \\
        & c_1^*(u) & b_1^*(u) &\\
        & b_2^*(u) & c_2^*(u) &\\
        & & & a_2^*(u)
    \end{pmatrix} = \begin{pmatrix}
        a_1^*(u) & & & \\
        & c_2(u) & -b_1(u) &\\
        & -b_2(u) & c_1(u) &\\
        & & & a_2^*(u)
    \end{pmatrix},
\]
where $a_1^*(u)$ and $a_2^*(u)$ are defined by 
\[
    c_1(u)c_2(u) - b_1(u)b_2(u) = a_1(u)a_1^*(u) = a_2(u)a_2^*(u).
\]
Note that if $u$ is projective, then $u^* = u^{-1}$. 

Let $V \cong \mathbb{C}^2$. For $u,w,v \in \End(V \otimes V)$, define the \textit{Yang-Baxter commutator} on $V \otimes V \otimes V$:
\begin{equation}\label{eq:ybcommutator}
    [[u,w,v]] = (u \otimes 1)(1 \otimes w)(v \otimes 1) - (1 \otimes v)(w \otimes 1)(1 \otimes u).
\end{equation}

Then the \textit{Yang-Baxter equation} is given by 
\begin{equation}\label{eq:ybe}
    [[u,w,v]] = 0, \quad u,w,v \in \End(V \otimes V).
\end{equation}

Let $S$ be the set of the six-vertex matrices with $c_1,c_2 \neq 0$. Note that since the six-vertex matrices are invertible, we have $a_1,a_2 \neq 0$. Let $S^\times$ be the subset with $b_1,b_2 \neq 0$. 

We now give the complete solution to the Yang-Baxter equation for the six-vertex matrices.

\begin{theorem}\label{thm:ybequation}
    Let $u,w,v \in S$ be the six-vertex matrices. Then the Yang-Baxter equation $[[u,w,v]] = 0$ holds if and only if one of the following equivalent cases holds. Moreover, the solutions are unique up to a constant.
    
    \begin{enumerate}
        \item Solving for $w$ in terms of $u$ and $v$. If the following condition holds:
        \begin{gather*}
            \begin{aligned}\label{eq:conditionforybcomposition}
                (a_1^*(u)-a_2(u))b_1(v) &= (a_2^*(v)-a_1(u))b_1(u),\\
                (a_1^*(v)-a_2(v))b_2(u) &= (a_2^*(u)-a_1(u))b_2(v),
            \end{aligned}
        \end{gather*}
        then $w$ is given by 
        \begin{align*}
            a_1(w) &= a_1(u)a_1(v) - b_1(v)b_2(u),\\
            a_2(w) &= a_2(u)a_2(v) - b_1(u)b_2(v),\\
            b_1(w) &= a_1^*(u)b_1(v) - a_1(v)b_1^*(u)\\
            &= a_2^*(v)b_1(u) - a_2(u)b_1^*(v),\\
            b_2(w) &= a_1^*(v)b_2(u) - a_1(u)b_2^*(v)\\
            &= a_2^*(u)b_2(v) - a_2(v)b_2^*(u),\\
            c_1(w) &= c_1(u)c_1(v),\\
            c_2(w) &= c_2(u)c_2(v).
        \end{align*}
        The condition ensures that the expressions for $b_1(w),b_2(w)$ agree.
        
        \item Solving for $u$ in terms of $w$ and $v$. If the following condition holds:
        \begin{gather*}
            \begin{aligned}
                (a_2(w)-a_1^*(w))b_1(v) &= (a_2(v)-a_1^*(v))b_1(w),\\
                (a_2(v)-a_1^*(w))b_2(v) &= (a_2^*(v)-a_1(v))b_2(w).
            \end{aligned}
        \end{gather*}
        
        then $u$ is given by 
        
        \begin{align*}
            a_1(u) &= a_1(w)a_1^*(v) + b_1(v)b_2(w),\\
            a_2(u) &= a_2(w)a_2^*(v) + b_1(w)b_2(v),\\
            b_1(u) &= b_1(w)a_2(v) - b_1(v)a_2(w),\\
            &= b_1(w)a_1^*(v) - b_1(v)a_1^*(w),\\
            b_2(u) &= a_1(v)b_2(w) - a_1(w)b_2(v),\\
            &= a_2^*(v)b_2(w) - a_2^*(w)b_2(v),\\
            c_1(u) &= c_1(w)c_2(v),\\
            c_2(u) &= c_1(v)c_2(w).
        \end{align*}
        The condition ensures that the expressions for $b_1(u),b_2(u)$ agree.
        
        \item Solving for $v$ in terms of $u$ and $w$. If the following condition holds:
        \begin{gather*}
            \begin{aligned}
                (a_1(u)-a_2^*(u))b_1(w) &= (a_1(w)-a_2^*(w))b_1(u),\\
                (a_2(u)-a_1^*(u))b_2(w) &= (a_2(w)-a_1^*(w))b_2(u).
            \end{aligned}
        \end{gather*}
        
        then $v$ is given by
        \begin{align*}
            a_1(v) &= a_1(w)a_1^*(u) + b_1(w)b_2(u),\\
            a_2(v) &= a_2(w)a_2^*(u) + b_1(u)b_2(w),\\
            b_1(v) &= b_1(w)a_1(u) - b_1(u)a_2(w),\\
            &= b_1(w)a_2^*(u) - b_1(u)a_2^*(w),\\
            b_2(v) &= a_2(u)b_2(w) - a_2(w)b_2(u),\\
            &= a_1^*(u)b_2(w) - a_1^*(w)b_2(u),\\
            c_1(v) &= c_1(w)c_2(u),\\
            c_2(v) &= c_1(u)c_2(w).
        \end{align*}
        The condition ensures that the expressions for $b_1(v),b_2(v)$ agree.
    \end{enumerate}
\end{theorem}
\begin{proof}
    The proof is in \Cref{section:ybproof}.
\end{proof}

\begin{remark}
    \Cref{thm:ybequation} slightly extends Theorem 1 of \cite{BBF11}, where the Yang-Baxter equation was solved for six-vertex matrices with non-zero entries $b_1,b_2$. We also give compact formulas for the solution in terms of the dual matrices, which are useful in applications.
\end{remark}

For the six-vertex matrices $S^\times$, we can reformulate the theorem in terms of the statistics $\Delta_1,\Delta_2\colon S^{\times} \to \mathbb{C}$ defined by
\begin{align*}
    \Delta_1(w) &=  \frac{a_1(w)a_2(w)+b_1(w)b_2(w)-c_1(w)c_2(w)}{2a_1(w)b_1(w)} = \frac{a_2(w) - a_1^*(w)}{2b_1(w)},\\
    \Delta_2(w) &= \frac{a_1(w)a_2(w)+b_1(w)b_2(w)-c_1(w)c_2(w)}{2a_2(w)b_2(w)} = \frac{a_1(w)-a_2^*(w)}{2b_2(w)}.
\end{align*}
Note that since the statistics are homogeneous in the coefficients, they are well-defined on projective six-vertex matrices. Note that the statistics $\Delta_1,\Delta_2$ are only defined for matrices with $b_1 \neq 0$ and $b_2 \neq 0$.

\begin{remark}
    The form of the statistics $\Delta_1,\Delta_2$ is chosen to match the existing literature. In the field-free case, Baxter \cite{Bax82} defined the statistics
    \[
        \Delta(w) \coloneqq \frac{a^2(w)+b^2(w)-c^2(w)}{2a(w)b(w)},
    \]
    where $a_1=a_2=a$, $b_1=b_2=b$, and $c_1=c_2=c$. In \cite{BBF11}, authors extended $\Delta$ to two statistics $\Delta_1, \Delta_2$ that we use here.
\end{remark}

Now we can reformulate \Cref{thm:ybequation} in terms of the statistics $\Delta_1,\Delta_2$ extending Theorems 1,2 of \cite{BBF11}.

\begin{corollary}
    Let $u,w,v \in S^\times$ be the six-vertex matrices. Then the Yang-Baxter equation $[[u,w,v]] = 0$ holds if and only if one of the following equivalent cases holds.
    \begin{enumerate}
        \item $\Delta_1(u) = \frac{a_2(v)}{a_1(v)}\Delta(v)$, $\Delta_2(u) = \frac{a_1(v)}{a_2(v)}\Delta_2(v)$, and $w$ is given by \Cref{thm:ybequation}
        \item $\Delta_1(w) = \Delta_1(v)$ and $\Delta_2(w) = \Delta_2(v)$, and $u$ is given by \Cref{thm:ybequation}
        \item $\frac{a_2(u)}{a_1(u)}\Delta_1(u) = \Delta_1(w)$, $\frac{a_1(u)}{a_2(u)}\Delta_2(u) = \Delta_2(w)$, and $v$ is given by \Cref{thm:ybequation}
    \end{enumerate}
\end{corollary}
\begin{proof}
    It follows from \Cref{thm:ybequation} and the fact that for any $u \in S^\times$, we have
    \begin{align*}
        \Delta_1(u^{-1}) = \frac{a_2(u)}{a_1(u)}\Delta_1(u), \quad 
        \Delta_2(u^{-1}) = \frac{a_1(u)}{a_2(u)}\Delta_2(u),
    \end{align*}
    which is a direct calculation.
\end{proof}

We now study the properties of the solutions of the Yang-Baxter equations. It will be natural to consider two classes of the six-vertex matrices: 
\begin{enumerate}
    \item The \textit{free-fermionic six-vertex matrices} with $\Delta_1 = \Delta_2 = 0$;
    \item The \textit{constant-field six-vertex matrices} with $a_1 = a_2$.
\end{enumerate}

\begin{proposition}
    Let $A \subset S^\times$ be a set of the six-vertex matrices such that for any $u,v \in A$, there exists $w \in A$ such that $[[u,w,v]] = 0$. Then $A$ is either the set of the free-fermionic six-vertex matrices or the set of the constant-field six-vertex matrices.
\end{proposition}
\begin{proof}
    Consider $[[u,w,u]] = 0$. Then $\Delta_1(u) = \Delta_1(u^{-1}) = \frac{a_2(u)}{a_1(u)}\Delta_1(u)$. Then either $\Delta_1(u) = 0$, but then $\Delta_2(u) = 0$, and $\Delta_1(v),\Delta_2(v) = 0$ for any other $v \in A$. Otherwise, $\Delta_1(u),\Delta_2(u) \neq 0$, and then $a_1(u) = a_2(u)$ for all $u \in A$.
\end{proof}

In other words, the free-fermionic and constant-field matrices are the only classes of matrices that form sets closed under the Yang-Baxter equation. 

Generally speaking, the equation. $[[u,w,v]] = 0$ and $[[v,w',u]] = 0$ do not imply that $w = w'$. However, many important solutions coming from the representation theory of the affine quantum (super)-groups have the property because they are parametrized by commutative groups. 

\begin{definition}\label{def:commutative}
    Let $u,w,v \in S$. We say that $u,w,v$ is a \textit{commutative solution} of the Yang-Baxter equation if $[[u,w,v]] = 0$ and $[[v,w,u]] = 0$.
\end{definition}

\begin{proposition}\label{lem:commutativity}
    The six-vertex matrices $u,w,v \in S$ form a commutative solution if and only if
    \begin{align*}
        b_1(u)b_2(v) &= b_1(v)b_2(u),\\
        (a_1(u)-a_1^*(u))b_1(v) &= (a_1(v)-a_1^*(v))b_1(u),\\
        (a_2(u)-a_2^*(u))b_2(v) &= (a_2(v)-a_2^*(v))b_2(u).
    \end{align*}
\end{proposition}
\begin{proof}
    Follows from the explicit formulas from \Cref{thm:ybequation}:
    \begin{align*}
        a_1(w) &= a_1(u)a_1(v) - b_1(v)b_2(u) = a_1(v)a_1(u) - b_1(u)b_2(v) = a_1(w'),\\
        a_2(w) &= a_2(u)a_2(v) - b_1(u)b_2(v) = a_2(v)a_2(u) - b_1(v)b_2(u) = a_2(w'),\\
        b_1(w) &= a_1^*(u)b_1(v) + a_1(v)b_1(u) = a_1^*(v)b_1(u) + a_1(u)b_1(v) = b_1(w'),\\
        b_2(w) &= a_2^*(u)b_2(v) + a_2(v)b_2(u) = a_2^*(v)b_2(u) + a_2(u)b_2(v) = b_2(w'),\\
        c_1(w) &= c_1(u)c_1(v) = c_1(w'),\\
        c_2(w) &= c_2(u)c_2(v) = c_2(w').\qedhere
    \end{align*}
\end{proof}

Let $\Delta^{\ff},\Delta^{\cf}$ be new simplified statistics on $S^\times$ defined by 
\[
    \Delta^{\ff}(u) = \frac{a_1(u) - a_2(u)}{2b_1(u)}, \quad \Delta^{\cf}(u) = \frac{a_1(u)-a_1^*(u)}{2b_1(u)}.
\]
\begin{corollary}
    Let $u,w,v \in S^\times$. Then $[[u,w,v]] = 0$ and $[[v,w,u]] = 0$ hold at the same time if and only if one of the following cases holds: 
    \begin{enumerate}
        \item (Commutative free-fermionic case) for both $u,v$, we have $a_1^* = a_2$, $a_2^* = a_1$, $b_2 = t\, b_1$ for some $t \in \mathbb{C}^\times$, and $\Delta^{\ff}(u) = \Delta^{\ff}(v)$.
        \item (Commutative constant-field case) for both $u,v$, we have $a_1 = a_2$, $a_1^* = a_2^*$, $b_2 = t\, b_1$ for some $t \in \mathbb{C}^\times$, and $\Delta^{\cf}(u) = \Delta^{\cf}(v)$.
    \end{enumerate}
\end{corollary}
\begin{proof}
    It is a direct check that both cases satisfy the conditions of \Cref{thm:ybequation} and \Cref{lem:commutativity}. 
    
    Conversely, suppose that both equations hold. If $a_1^*(u) = a_2(u)$ and $a_1^*(v) = a_2(v)$, then $\Delta_{i}(x) = 0$ for $i=1,2$, $x=u,v$, and the equations have a solution $w \in S^\times$. Then by \Cref{lem:commutativity}, we have $b_2(u)=t\,b_1(u)$ and $b_2(v) = t\,b_1(u)$, and
    \begin{align*}
        \Delta^{\ff}(u) = \frac{a_1(u) - a_2(u)}{b_1(u)} = \frac{a_1(v)-a_2(v)}{b_1(v)} = \Delta^{\ff}(v).
    \end{align*}
    
    If $a_1^*(u) \neq a_2(u)$ and $a_2^*(u) \neq a_1(u)$, then denote $\Delta_2(u)/\Delta_1(u) = b_2(u)/b_1(u) = t$. Then from \Cref{thm:ybequation} and \Cref{lem:commutativity}, it follows that $\Delta^{\cf}(u) = \Delta^{\cf}(v)$, together with $a_1 = a_2$ and $a_1^* = a_2^*$.
\end{proof}

\Cref{thm:ybequation} and \Cref{lem:commutativity} help us define various classes of the six-vertex matrices:

\begin{enumerate}
    \item The \textit{diagonal matrices} with $b_1,b_2 = 0$;
    \item The \textit{five-vertex matrices} with $b_1=0, b_2 \neq 0$ or $b_1 \neq 0, b_2 = 0$;
    \item The \textit{free-fermionic matrices} with $a_1^* = a_2$ and $a_2^* = a_1$;
    \item The \textit{non-free-fermionic matrices} with $a_2^* \neq a_1$ and $a_1^* \neq a_2$;
    \item The \textit{constant-field} matrices with $a_1 = a_2$;
    \item The \textit{non-constant-field} matrices with $a_1 \neq a_2$;
    \item The \textit{field-free matrices} with $a_1 = a_2$, $b_1 = b_2$, and $c_1 = c_2$;
    \item The \textit{degenerate matrices} with $a_1 = a_2 = a_1^* = a_2^*$.
\end{enumerate}

\begin{remark}
    The names free-fermionic and constant-field come from identifying the six-vertex models with XXZ spin chains where the free-fermionic condition ensures that the fermions do not interact. See \cite{FW70}, where the term ``free-fermionic'' first appeared in the context of the eight-vertex models.
\end{remark}

Let $u,v$ be two free-fermionic matrices from the same commutative family. We want a convenient parametrization for the entire family of commutative free-fermionic matrices. By the results above, it means that we have both $b_2(u)/b_1(u) = b_2(v)/b_1(v) = t$ and $\Delta^{\ff}(u) = \Delta^{\ff}(v)$. Consider the free-fermionic quadric
\[
    a_1a_2 + b_1b_2 = c_1c_2.
\]
Let $t = q_1q_2$ and $\Delta^{\ff} = \frac{q_1+q_2}{2}$. Then we have
\begin{align*}
    a_1^2 - 2\Delta^{\ff}a_1b_1 + t\,b_1^2 = c_1c_2\\
    (a_1 - q_1b_1)(a_1 - q_2b_2) = c_1c_2.
\end{align*}

Then we can parametrize the commutative free-fermionic weights as follows:
\begin{equation}\label{eq:ffparametrization}
    \begin{aligned}
        a_1(v) &= q_1 z_1 - q_2 z_2,\\
        a_2(v) &= q_1 z_2 - q_2 z_1,
    \end{aligned} \quad
    \begin{aligned}
        b_1(v) &= q_1(z_1-z_2)\beta,\\
        b_2(v) &= q_2(z_1-z_2)\beta^{-1},
    \end{aligned} \quad
    \begin{aligned}
        c_1(v) &= z_1(q_1-q_2)w,\\
        c_2(v) &= z_2(q_1-q_2)w^{-1}
    \end{aligned}.
\end{equation}

Similarly, let $u,v$ be two constant-field matrices from the same commutative family. By above, we have $b_2(u)/b_1(u) = b_2(v)/b_1(v) = t$, and $\Delta^{\cf}(u) = \Delta^{\cf}(v)$. Then consider the definition of the dual elements: 
\[
    a_1a_1^* + b_1b_2 = c_1c_2.
\]
Then set $t = q_1q_2$ and $\Delta^{\cf}(u) = \Delta^{\cf}(v) = \frac{q_1+q_2}{2}$. Then we have
\begin{align*}
    a_1^2 - 2\Delta^{\cf}a_1b_1 + t\,b_1^2 = c_1c_2\\
    (a_1 - q_1b_1)(a_1 - q_2b_2) = c_1c_2.
\end{align*}

Then we can parametrize the commutative constant-field weights as follows:
\begin{equation}\label{eq:cfparametrization}
    \begin{aligned}
        a_1(v) &= q_1 z_1 - q_2 z_2,\\
        a_2(v) &= q_1 z_2 - q_2 z_1,
    \end{aligned} \quad
    \begin{aligned}
        b_1(v) &= q_1(z_1-z_2)\beta,\\
        b_2(v) &= q_2(z_1-z_2)\beta^{-1},
    \end{aligned} \quad
    \begin{aligned}
        c_1(v) &= z_1(q_1-q_2)w,\\
        c_2(v) &= z_2(q_1-q_2)w^{-1}
    \end{aligned}.
\end{equation}

In the next section, we use these expressions to give parametrized solutions with commutative parameter groups. Moreover, the discussion shows that these families are the largest commutative parametrizations.

\section{Parametrized Yang-Baxter equation}\label{sec:pybe}
Let $G$ be a group, and let $\pi\colon G \to \End(V \otimes V)$. We say that we have a \textit{parametrized solution of the Yang-Baxter equation} when we have 
\[
    [[\pi(g), \pi(gh), \pi(h)]] = 0, \quad g,h \in G.
\]

Let us first review the parametrized Yang-Baxter equation for the free-fermionic matrices by a non-commutative group. This result first appeared in \cite{KBI93}(page 126) as ``another curious solution''. The original parameter group was $\SL_2(\mathbb{C})$. In Theorem 2 of \cite{BBF11}, the authors extended the parametrization to the non-commutative group $\GL_2(\mathbb{C}) \times \GL_1(\mathbb{C})$. These are the only examples of a parametrized Yang-Baxter equation by a non-commutative group. See \cite{B16} for the FRT construction of the free-fermionic bialgebra built on this parametrized solutions. 

Let the parametrization map $\tau\colon \GL_2(\mathbb{C}) \times \GL_1(\mathbb{C}) \to S$ be defined by 
\begin{equation}\label{eq:tau}
    \tau\left(\begin{pmatrix}
        a_1 & b_2\\
        -b_1 & a_2
    \end{pmatrix} \times (c_1)\right) \to \begin{pmatrix}
        a_1 & & & \\
        & c_1 & b_1 &\\
        & b_2 & c_2 &\\
        & & & a_2
    \end{pmatrix},
\end{equation}
where $c_2 = (a_1a_2 + b_1b_2)/c_1$. Note that the resulting matrices are free-fermionic as they by design satisfy $a_1a_2 + b_1b_2 = c_1c_2$.

\begin{theorem}[Theorem 3 in \cite{BBF11}]
    We have the parametrized Yang-Baxter equation
    \[
        [[\tau(g), \tau(gh), \tau(h)]] = 0, \quad g,h \in \GL_2(\mathbb{C}) \times \GL_1(\mathbb{C}).
    \] 
\end{theorem}
\begin{proof}
    It follows from the explicit formulas in \Cref{thm:ybequation}.
\end{proof}

Let $q_1,q_2 \in \mathbb{C}, \beta \in \mathbb{C}^\times$ be constants that we think of as global parameters. Let 
\[
R^{\cf}_{q_1,q_2; \beta}, R^{\ff}_{q_1,q_2; \beta} \colon (\mathbb{C}^\times)^3 \to S
\]
be the parametrization maps defined by 
\begin{align}\label{eq:rmatrices}
    R^{\cf}_{q_1,q_2; \beta}(z_1,z_2; w) &= \begin{pmatrix}
        q_1 z_1 - q_2 z_2 & & & \\
        & z_1(q_1-q_2)w & q_1(z_1 - z_2)\beta &\\
        & q_2(z_1-z_2)\beta^{-1} & z_2(q_1-q_2)w^{-1} &\\
        & & & q_1 z_1 - q_2 z_2
    \end{pmatrix}, \\
    R^{\ff}_{q_1,q_2; \beta}(z_1,z_2; w) &= \begin{pmatrix}
        q_1 z_1 - q_2 z_2 & & & \\
        & z_1(q_1-q_2)w & q_1(z_1 - z_2)\beta &\\
        & q_2(z_1-z_2)\beta^{-1} & z_2(q_1-q_2)w^{-1} &\\
        & & & q_1 z_2 - q_2 z_1
    \end{pmatrix}.
\end{align}

Note that we used the explicit parametrization of the largest commutative families of the free-fermionic and constant-field matrices from \eqref{eq:ffparametrization} and \eqref{eq:cfparametrization}. We also note that the only difference is in the entry $a_2$ in both matrices. 

\begin{proposition}
    We have the parametrized Yang-Baxter equations
    \begin{align}
        [[R^{\cf}_{q_1,q_2;\beta}(g),R^{\cf}_{q_1,q_2;\beta}(gh),R^{\cf}_{q_1,q_2;\beta}(h)]], \quad g,h \in (\mathbb{C}^\times)^3,\\
        [[R^{\ff}_{q_1,q_2;\beta}(g),R^{\ff}_{q_1,q_2;\beta}(gh),R^{\ff}_{q_1,q_2;\beta}(h)]], \quad g,h \in (\mathbb{C}^\times)^3.
    \end{align}
\end{proposition}
\begin{proof}
    Direct calculation.
\end{proof}

We now show that many known parametrized solutions of the Yang-Baxter equations are special cases of this maximal commutative parametrization. 

\begin{example}[The $R$-matrices of affine quantum (super)-groups]\label{example:quantumgroups}
    Let $\beta = 1$, $w = 1$, $q_1 = q$, $q_2 = q^{-1}$, $z_1 = z$, and $z_2 = z^{-1}$. Then the six-vertrx matrices $R^{\cf}_{q}(z) \coloneqq R_{q,q^{-1};1}(z,z^{-1};1)$ match the $R$-matrices for the evaluation modules of the affine quantum group $U_q(\widehat{\mathfrak{sl}}(2))$:
    \[
        R^{\cf}_q(z) = \begin{pmatrix}
            q z - q^{-1} z^{-1} & & & \\
            & z(q-q^{-1}) & q(z - z^{-1}) &\\
            & q^{-1}(z-z^{-1}) & z^{-1}(q-q^{-1}) &\\
            & & & q z - q^{-1} z^{-1}
        \end{pmatrix}.
    \]
    
    Similarly, the six-vertrx matrices $R^{\ff}_{q}(z) \coloneqq R^{\ff}_{q,q^{-1};1}(z,z^{-1};1)$ match the $R$-matrices for the evaluation modules of the affine quantum supergroup $U_q(\widehat{\mathfrak{sl}}(1|1))$:
    \[
        R^{\ff}_q(z) = \begin{pmatrix}
            q z - q^{-1} z^{-1} & & & \\
            & z(q-q^{-1}) & q(z - z^{-1}) &\\
            & q^{-1}(z-z^{-1}) & z^{-1}(q-q^{-1}) &\\
            & & & q z^{-1} - q^{-1} z
        \end{pmatrix}.
    \]
    
    The freedom of choosing $\beta$ comes from the Drinfeld twist of the corresponding quantum group. The freedom in choosing the parameter $w$ comes from the change of basis in the evaluation modules. 
\end{example}

\begin{example}[Five-vertex matrices]\label{example:fivevertex}
    Let $q_1 = 0$ or $q_2 = 0$. Then the resulting degeneration does not depend on the second $q$-parameter. After rescaling, we get 
    \begin{align*}
        R^{\cf,\,b_1=0}_{\beta} &= \begin{pmatrix}
            z_1 & & & \\
            & z_1 w & (z_1 - z_2)\beta &\\
            & 0 & z_2w^{-1} &\\
            & & & z_1
        \end{pmatrix}, \quad 
        R^{\cf,\,b_2=0}_{\beta} &= \begin{pmatrix}
            z_2 & & & \\
            & z_1w & 0 &\\
            & (z_1-z_2)\beta & z_2w^{-1} &\\
            & & & z_2
        \end{pmatrix}.\\
    \end{align*}
    \begin{align*}
        R^{\ff,\,b_1=0}_{\beta} &= \begin{pmatrix}
            z_1 & & & \\
            & z_1 w & (z_1 - z_2)\beta &\\
            & 0 & z_2w^{-1} &\\
            & & & z_2
        \end{pmatrix}, \quad 
        R^{\ff,\,b_2=0}_{\beta} &= \begin{pmatrix}
            z_2 & & & \\
            & z_1w & 0 &\\
            & (z_1-z_2)\beta & z_2w^{-1} &\\
            & & & z_1
        \end{pmatrix}.\\
    \end{align*}
    
    When $w = 1$ and $\beta = 1$, the five-vertex constant-field matrices can be considered as the combinatorial $R$-matrices of the crystallization for $U_q(\widehat{\mathfrak{sl}}(2))$ and $U_q(\widehat{\mathfrak{sl}}(1|1))$ when formally set $q \to 0$. 
\end{example}

\begin{example}
    Consider $R_{q_1,q_2; \beta}^{\cf}(z_1,z_2; w)$ with $q_1 = q$, $q_2 = q^{-1}$, $\beta = 1$, and $w = 1$. Then
    \[
        R_{q}^{\cf}(z_1,z_2) =  (q z_1 - q^{-1} z_2)I + (z_1-z_2)E,
    \]
    where $E$ is given by 
    \[
        \begin{pmatrix}
            0 & 0 & 0 & 0\\
            0 & -q^{-1} & q & 0\\
            0 & q^{-1} & -q & 0\\
            0 & 0 & 0 & 0
        \end{pmatrix}.
    \]
    
    Let $V \cong \mathbb{C}^2$. Consider $V^{\otimes n}$. Let $E_k$ be the operator on $V^{\otimes n}$ acting on $k$-th and $(k+1)$-th component by $E$, and by identity elsewhere. Then the elements $E_k$ form the Temperley-Lieb algebra with $\delta = -(q+q^{-1})$, that is, they satisfy the following relations:
    \begin{enumerate}
        \item $E_k^2 = -(q_1+q_2)E_k$
        \item $E_k E_{k+1} E_k = E_k$
        \item $E_{k+1} E_k E_{k+1} = E_{k+1}$
        \item $E_i E_j = E_j E_i$ when $|i-j|>1$
    \end{enumerate}
    
    Then the matrices $(R_q^{\cf}(z_1,z_2))_k$ can be seen as elements in this Temperley-Lieb algebra.
\end{example}

\begin{example}
    Set $q_1 = q$, $q_2 = q^{-1}$. Set $b = q^{-1}$, $z_1 = z$, $z_2 = z^{-1}$, and $w = z^{-1}$. Let $\zeta_3$ be the primitive cube of $-1$. Then set $q = \zeta_3$ and $z = \zeta_3$. Then the matrix $R_{q_1,q_2;\beta}^{\cf}(z_1,z_2; w)$ becomes
    \[
        \begin{pmatrix}
        1 & & & \\
        & 1 & 1 &\\
        & 1 & 1 &\\
        & & & 1
    \end{pmatrix},
    \]
    which is used in the enumeration of the alternating sign matrices. See \cite{RPZ12} and citations therein. 
\end{example}

\section{Yang-Baxter Groupoid}
The Yang-Baxter equation can be used to define a partial operation on the set of the projective six-vertex matrices. In this section, we study the properties of this operation and find that the set of the non-free-fermionic matrices with the Yang-Baxter composition forms a groupoid. Thanks to uniqueness, \Cref{thm:ybequation} gives a well-defined partial binary operation
\[
    *\colon \PP(S^\times) \times \PP(S^\times) \to \PP(S^\times)
\]
given by $(u, v) \mapsto w = u*v \in \PP(S^\times)$ whenever $[[u,w,v]] = 0$. We call it \textit{the Yang-Baxter composition}. Note that if $w \not\in \PP(S^\times)$, we leave the composition $u*v$ undefined. In particular, despite $[[u,1,u^{-1}]] = 0$, we \textbf{do not} set $u*u^{-1} = [I]$ or $u^{-1}*u = [I]$ because $[I] \not\in \PP(S^\times)$.

\subsection{Structure of the Yang-Baxter equation}
In this section we consider the general structure of the solutions of the Yang-Baxter equations. We will apply the results of this section for the six-vertex matrices in the next section. Note that the results are applicable to any solutions of the Yang-Baxter equations, not only for the six-vertex matrices.

Let $V \cong \mathbb{C}^n$. For $u,w,v \in \End(V \otimes V)$. Recall that the Yang-Baxter commutator on $V \otimes V \otimes V$ is given by
\begin{equation}
    [[u,w,v]] = (u \otimes 1)(1 \otimes w)(v \otimes 1) - (1 \otimes v)(w \otimes 1)(1 \otimes u).
\end{equation}

\begin{lemma}
    Let $u,v,w \in \End(V \otimes V)$ be invertible endomorphisms. Then
    \begin{align*}
        [[u,w,v]](1 \otimes u^{-1}) + (u \otimes 1)[[u^{-1},v,w]] &= 0,\\
        [[u,w,v]](v^{-1} \otimes 1) + (1 \otimes v)[[w,u,v^{-1}]] &= 0.
    \end{align*}
\end{lemma}
\begin{proof}
    Direct calculation: 
    \begin{multline*}
        (u \otimes 1)(1 \otimes w)(v \otimes 1)(1 \otimes u^{-1}) - (1 \otimes v)(w \otimes 1)(1 \otimes u)(1 \otimes u^{-1}) + \\
        + (u \otimes 1)(u^{-1} \otimes 1)(1 \otimes v)(w \otimes 1) - (u \otimes 1)(1 \otimes w)(v \otimes 1)(1 \otimes u^{-1}) = 0.
    \end{multline*}
    The second identity is proved analogously. 
\end{proof}

\begin{corollary}\label{lem:equivalentyb}
    Let $u,w,v \in \End(V \otimes V)$ be invertible endomorphisms. Then the following Yang-Baxter equations are equivalent:
    \begin{align*}
        [[u,w,v]] = 0 &\iff [[w,u,v^{-1}]] = 0 \iff [[u^{-1}, v, w]] = 0 \iff\\ [[v^{-1},w^{-1},u^{-1}]] = 0 &
        \iff [[v,u^{-1},w^{-1}]] = 0 \iff [[w^{-1},v^{-1},u]] = 0.
    \end{align*}
\end{corollary}
\begin{proof}
    Repeated application of the previous lemma.
\end{proof}

Let $I \in \End(V \otimes V)$ be the identity endomorphism. 
\begin{lemma}
    Let $u \in \End(V \otimes V)$ be an invertible endomorphism. We have $[[I,u,u]] = 0$ and $[[u,u,I]] = 0$ and $[[u,I,u^{-1}]] = 0$ and $[[u^{-1},I,u]] = 0$.
\end{lemma}
\begin{proof}
    Direct calculation:
    \[
        [[I,u,u]] = (I \otimes 1)(1 \otimes u)(u \otimes 1) - (1 \otimes u)(u \otimes 1)(1 \otimes I) = 0.
    \]
    \[
        [[u,I,u^{-1}]] = (u \otimes 1)(1 \otimes I)(u^{-1} \otimes 1) - (1 \otimes u^{-1})(I \otimes 1)(1 \otimes u) = 0.
    \]
    Other identities are proved analogously.
\end{proof}

\subsection{Construction of the non-free-fermionic groupoid}\label{sec:groupoid}
Recall that an (algebraic) groupoid is a set $G$ with a unary operation $^{-1}\colon G \to G$, and a partial binary operation $G \times G \to G$ such that for all $a,b,c \in G$, we have
\begin{enumerate}
    \item (Associativity) If $a*b$ and $b*c$ are defined, then $(a*b)*c$ and $a*(b*c)$ are defined and equal. Conversely, if $(a*b)*c$ and $a*(b*c)$ are defined, then so are both $a*b$ and $b*c$, as well as $(a*b)*c = a*(b*c)$;
    \item (Inverse) $a^{-1} * a$ and $a * a^{-1}$ are always defined; 
    \item (Identity) If $a*b$ is defined, then $a*b*b^{-1} = a$ and $a^{-1}*a*b = b$. 
\end{enumerate}

It follows from above that $(a^{-1})^{-1} = a$ and if $a*b$ is defined, then $(a*b)^{-1} = b^{-1} * a^{-1}$. 

Recall that $S^\times$ be the set of six-vertex matrices with non-zero entries $a_1,a_2,b_1,b_2,c_1,c_2$. Recall that $\Delta = (\Delta_1,\Delta_2)$, and the statistics $\Delta_1,\Delta_2$ are defined in the previous section. Let $S_{\nff}$ be the subset set of \textit{non-free-fermionic matrices}, that is, the six-vertex matrices $u \in S^\times$ with $\Delta(u) \neq (0,0)$. The goal of this section is to construct the non-free-fermionic groupoid. 

Let $G$ be the set of pairs $(u, \Delta(u)) \in S_{\nff} \times (\mathbb{C}^\times)^2$ together with the set of identities $(I, (\delta_1,\delta_2))$ for all $\delta_1,\delta_2 \in (\mathbb{C}^\times)^2$. Define an unary operation $^{-1}\colon G \to G$ by $(u, \delta(u)) \mapsto (u^{-1}, \delta(u^{-1}))$, where $u^{-1}$ is the usual matrix inverse. Define a partial binary operation $*\colon G \times G \to G$ by 
\begin{align*}
    (u,\Delta(u)) * (v, \Delta(v)) &= (u*v, \Delta(u*v)),
\end{align*}
where $u*v$ is the Yang-Baxter composition, and the composition is defined only when $\Delta(u) = \Delta(v^{-1})$. We also set
\begin{align*}
    (I, \Delta(v^{-1}))*(v, \Delta(v)) &= (v, \Delta(v)),\\
    (u, \Delta(u))*(I, \Delta(u)) &= (u, \Delta(u)).
\end{align*}

\begin{lemma}\label{lem:associativity}
    The Yang-Baxter composition is associative in $S_{\nff}$: if $u*v$ and $v*w$ are defined, then $u*(v*w)$ and $(u*v)*w$ are defined, and they are equal.
\end{lemma}
\begin{proof}
    Since $u*v$ and $v*w$ are defined. Then by \Cref{lem:equivalentyb} and \Cref{thm:ybequation}, 
    \[
        \Delta(u) = \Delta(v^{-1}) = \Delta((v*w)^{-1}) \quad \text{and} \quad \Delta(u*v) = \Delta(v) = \Delta(w^{-1}),
    \]
    hence, $u*(v*w)$ and $(u*v)*w$  are defined. Now we show that they are equal using the explicit formulas for the coefficients. The coefficients $c_1,c_2$ are easy. For $i=1,2$, we have
    \begin{align*}
        c_i((u*v)*w) = c_i(u*v)c_1(w) = c_i(u)c_i(v)c_i(w) = c_i(u)c_i(v*w) = c_i(u*(v*w)).
    \end{align*}
    
    For coefficients $b_1,b_2$, we use the explicit weights from \Cref{thm:ybequation}. For $i=1, j=2$ or $i=2,j=1$ we have
    \begin{align*}
        b_i((u*v)*w) &= a_i((u*v)^{-1})b_i(w) + a_i(w)b_i(u*v)\\
        &= (a_i(v^{-1})a_1(u^{-1})-b_i(u)b_j(v))b_i(w) + a_i(w)(a_1(u^{-1})b_1(v)+a_i(v)b_1(u))\\
        &= a_i(u^{-1})(a_i(v^{-1})b_i(w)+a_i(w)b_i(v)) + (a_i(v)a_i(w)-b_i(w)b_j(v))b_i(u)\\
        &= a_i(u^{-1})b_i(v*w) + a_i(v*w)b_i(u)\\
        &= b_i(u*(v*w)).
    \end{align*}
    
    Similarly, for coefficients $a_1$, we have
    \begin{align*}
        a_1((u*v)*w) &= a_1(u*v)a_1(w) + b_1(w)b_2(u*v)\\
        &= (a_1(u)a_1(v) + b_1(v)b_2(u))a_1(w) - b_1(w)(a_1(u)b_2(v) + a_1^*(v)b_2(u))\\
        &= a_1(w)a_1(u)a_1(v) - a_1(w)b_1(v)b_2(u) - a_1(u)b_1(w)b_2(v) - a_1^*(v)b_1(w)b_2(u)\\
        &= a_1(u)(a_1(v)a_1(w)-b_1(w)b_2(v)) - b_2(u)(a_1(w)b_1(v) + a_1^*(v)b_1(w))\\
        &= a_1(u)a_1(v*w) - b_1(v*w)b_2(u)\\
        &= a_1(u*(v*w)),
    \end{align*}
    and $a_2$ is dealt mutatis mutandis. Hence, the Yang-Baxter composition is associative.
\end{proof}

\begin{remark}
    In Section 2 of \cite{BBF11}, authors give an heuristic argument that a well-defined binary operation $u*v$ which solves the Yang-Baxter equation $[[u,u*v,v]] = 0$ is always associative. Our results can be seen as additional evidence to support the claim. Moreover, we prove something stronger: if $u*v$ and $v*w$ are defined, then so is $u*v*w$. 
\end{remark}

\begin{corollary}
    The set $G$ described above is a groupoid.
\end{corollary}
\begin{proof}
    We already proved all the pieces in the previous results.
    \begin{enumerate}
        \item (Associativity) is proved in \Cref{lem:associativity};
        \item (Inverse) It follows from \Cref{thm:ybequation};
        \item (Identity) It follows from \Cref{lem:equivalentyb}.
    \end{enumerate}
    Hence, the set $G$ is a groupoid. 
\end{proof}

Let $\pi\colon G \to S$ be the map from the groupoid to the non-free-fermionic six-vertex matrices defined by $\pi(u,\Delta(u)) = u$ by choosing any representative, and $\pi(I, (\delta_1,\delta_2)) = I$. By above, we get the following result.

\begin{theorem}
    We have the parametrized Yang-Baxter equation
    \[
        [[\pi(g), \pi(g*h), \pi(h)]] = 0, \quad g,h \in G_{\nff}.
    \]
\end{theorem}

\begin{remark}
    Formally, we can include the free-fermionic matrices to exntend $G_{\nff}$ to the groupoid of all six-vertex matrices $S^\times$, but by \Cref{thm:ybequation}, the free-fermionic matrices have the defined composition only with other free-fermionic matrices. Hence, they form an isolated vertex in the groupoid. The vertex group equals exactly $\GL_2(\mathbb{C}) \times \GL_1(\mathbb{C})$ as was shown in the previous section.
\end{remark}

\appendix

\section{Proof of \Cref{thm:ybequation}}\label{section:ybproof}

Let $u,w,v \in S$. We solve the Yang-Baxter equation $[[u,w,v]] = 0$ by solving equations for matrix coefficients:
\[
    [[u,w,v]]_{i_1,i_2,i_3}^{j_1,j_2,j_3} \coloneqq \langle e_{i_1} \otimes e_{i_2} \otimes e_{i_3} \mid [[u,w,v]] \mid e_{j_1} \otimes e_{j_2} \otimes e_{j_3} \rangle = 0.
\]

All of the equations are:
\begin{align*}
    c_1(w) c_2(u) c_2(v) &= c_1(u) c_1(v) c_2(w),\\
    a_1(w) c_1(u) c_1(v)+b_1(v) b_2(u) c_1(w)&=a_1(u) a_1(v) c_1(w),\\
    a_1(w) b_1(u) c_1(v)+b_1(v) c_1(w) c_2(u)&=a_1(u) b_1(w) c_1(v),\\
    a_1(w) b_2(v) c_1(u)+b_2(u) c_1(w) c_2(v)&=a_1(v) b_2(w) c_1(u),\\
    a_1(w) b_1(u) c_2(v)+b_1(v) c_1(u) c_2(w)&=a_1(u) b_1(w) c_2(v),\\
    a_2(w) c_1(u) c_1(v)+b_1(u) b_2(v) c_1(w)&=a_2(u) a_2(v) c_1(w),\\
    a_2(w) b_1(v) c_1(u)+b_1(u) c_1(w) c_2(v)&=a_2(v) b_1(w) c_1(u),\\
    a_1(w) b_2(v) c_2(u)+b_2(u) c_1(v) c_2(w)&=a_1(v) b_2(w) c_2(u),\\
    a_1(w) c_2(u) c_2(v)+b_1(v) b_2(u) c_2(w)&=a_1(u) a_1(v) c_2(w),\\
    a_2(w) b_2(u) c_1(v)+b_2(v) c_1(w) c_2(u)&=a_2(u) b_2(w) c_1(v),\\
    a_2(w) b_1(v) c_2(u)+b_1(u) c_1(v) c_2(w)&=a_2(v) b_1(w) c_2(u),\\
    a_2(w) b_2(u) c_2(v)+b_2(v) c_1(u) c_2(w)&=a_2(u) b_2(w) c_2(v),\\
    a_2(w) c_2(u) c_2(v)+b_1(u) b_2(v) c_2(w)&=a_2(u) a_2(v) c_2(w).
\end{align*}

We set $c_1(w) = c_1(u)c_1(v)$ and $c_2(w) = c_2(u)c_2(v)$. By making this choice, we choose a representative for $[w]$. After deleting duplicates, the equations become
\begin{align*}
    a_1(u) a_1(v)=a_1(w)+b_1(v) b_2(u),\\
    a_2(u) a_2(v)=a_2(w)+b_1(u) b_2(v),\\
    a_1(w) b_1(u)+b_1(v) c_1(u) c_2(u)=a_1(u) b_1(w),\\
    a_1(w) b_2(v)+b_2(u) c_1(v) c_2(v)=a_1(v) b_2(w),\\
    a_2(w) b_1(v)+b_1(u) c_1(v) c_2(v)=a_2(v) b_1(w),\\
    a_2(w) b_2(u)+b_2(v) c_1(u) c_2(u)=a_2(u) b_2(w).
\end{align*}
We get $a_1(w) = a_1(u)a_1(v) - b_1(v)b_2(u)$ and $a_2(w) = a_2(u)a_2(v) - b_1(u)b_2(v)$. The remaining equations become
\begin{align*}
    a_1(u) a_1(v) b_1(u)+b_1(v) c_1(u) c_2(u)=a_1(u) b_1(w)+b_1(u) b_1(v) b_2(u),\\
    a_1(u) a_1(v) b_2(v)+b_2(u) c_1(v) c_2(v)=a_1(v) b_2(w)+b_1(v) b_2(u) b_2(v),\\
    a_2(u) a_2(v) b_1(v)+b_1(u) c_1(v) c_2(v)=a_2(v) b_1(w)+b_1(u) b_1(v) b_2(v),\\
    a_2(u) a_2(v) b_2(u)+b_2(v) c_1(u) c_2(u)=a_2(u) b_2(w)+b_1(u) b_2(u) b_2(v).
\end{align*}

Recall that we define $c_1(v)c_2(v) - b_1(v)b_2(v) = a_1(v)a_1^*(v) = a_2(v)a_2^*(v)$ for any $v \in S$. Then the equations become 
\begin{align*}
    b_1(w) &= a_1^*(u)b_1(v) + a_1(v)b_1(u) = a_1^*(u)b_1(v) - a_1(v)b_1^*(u)\\
    &= a_2^*(v)b_1(u) + a_2(u)b_1(v) = a_2^*(v)b_1(u) - a_2(u)b_1^*(v),\\
    b_2(w) &= a_1^*(v)b_2(u) + a_1(u)b_2(v) = a_1^*(v)b_2(u) - a_1(u)b_2^*(v)\\
    &= a_2^*(u)b_2(v) + a_2(v)b_2(u)= a_2^*(u)b_2(v) - a_2(v)b_2(u^*).
\end{align*}
For consistency of the solution, expressions for $b_1(w)$ and $b_2(w)$ should agree. We write these conditions as follows:
\begin{align*}
    (a_1(v)-a_2^*(v))b_1(u) &= (a_2(u)-a_1^*(u))b_1(v),\\
    (a_1(u)-a_2^*(u))b_2(v) &= (a_2(v)-a_1^*(v))b_2(u).
\end{align*}

The proof of the other cases follows from \Cref{lem:equivalentyb}.

\section{Algebraic structure of the Yang-Baxter equation}
In this section we consider the general algebraic structure of the solutions of the Yang-Baxter equations. We will apply the results of this section for the six-vertex matrices in the next section. Note that the results are applicable to any solutions of the Yang-Baxter equations, not only for the six-vertex matrices.

Recall that a \textit{magmoid} is a set $M$ with a partial binary operation $*\colon M \times M \to M$. A magmoid is called \textit{quasiassociative} if the following holds: if $u*v$ and $v*w$ exist, then if either $(u*v)*w$ or $u*(v*w)$ exist, then so does the other, and the two are equal. A magmoid is called \text{associative} if the following holds: if $u*v$ and $v*w$ exist, then both $(u*v)*w$ and $u*(v*w)$ exist, and they are equal. Let $^{-1}\colon M \to M$ be an unary operation. A magmoid is called \textit{involution magmoid} if $(u^{-1})^{-1} = u$ and if $u*v$ exists, then $v^{-1}*u^{-1}$ exists and equal to $(u*v)^{-1}$. An involution magmoid is called \textit{invertible} if when $u*v$ exists, then $(u*v)*v^{-1} = u$ and $u^{-1}*(u*v) = v$. The invertible magmoid can also be called \textit{invertible quasigroupoid}. A magmoid is called \textit{unital} if there is a unique element $I \in M$ such that $u*I$ and $I*u$ are defined for all $u \in M$, and $u*I = u$ and $I*u = u$. 

Let $A \subset \End(V \otimes V)$ be a set with the following properties:
\begin{enumerate}
    \item If $u,v \in A$, there there exists a unique (up to a scalar multiple) solution $w \in A$ of the Yang-Baxter equation $[[u,w,v]] = 0$.
    \item Every $a \in A$ is invertible, and $a^{-1} \in A$.
\end{enumerate}

Thanks to uniqueness, we have a well-defined partial binary operation on the projective endomorphisms of $A$ that we call \textit{the Yang-Baxter composition}. It is a partial operation
\[
    *\colon \PP(A) \times \PP(A) \to \PP(A)
\]
given by $(u, v) \mapsto w = u*v \in \PP(A)$ whenever $[[u,w,v]] = 0$. See Section 2 of \cite{BBF11} for details. Note that if $w \not\in \PP(A)$, we leave the composition $u*v$ undefined. In particular, if $I \not\in A$, then despite $[[u,I,u^{-1}]] = 0$, we \textbf{do not} set $u*u^{-1} = [I]$ or $u^{-1}*u = [I]$ because $[I] \not\in \PP(A)$. 

\begin{proposition}
    The set $A$ satisfying the properties above forms an invertible magmoid $(A, *, ^{-1})$. If $I \in A$, then $A$ is unital as well. 
\end{proposition}
\begin{proof}
    We reformulate \Cref{lem:equivalentyb} in terms of the Yang-Baxter equation. If $u*v$ is defined, then we have the following identities:
    \begin{enumerate}
        \item $(u*v)^{-1} = v^{-1}*u^{-1}$
        \item $(u*v)*v^{-1} = u$
        \item $u^{-1}*(u*v) = v$
        \item $v*(u*v)^{-1} = u^{-1}$
        \item $(u*v)^{-1}*u = v^{-1}$
    \end{enumerate}
    Hence, $A$ is an invertible magmoid. 
\end{proof}

As we have shown in the previous section, the set of the six-vertex matrices $S^\times$ satisfies the conditions of the proposition. Moreover, $S^\times$ is associative, hence, forms an associative invertible magmoid. It is always possible to complete a (typically non-unital) invertible magmoid to make it a groupoid by adding the necessary units as we did in the previous section for the six-vertex matrices. We formulate the following conjecture that extends the conjecture from Section 2 of \cite{BBF11}.

\begin{conjecture}\label{conj:associativity}
    Let $A \subset \End(V \otimes V)$ be a set with the following properties:
    \begin{enumerate}
        \item If $u,v \in A$, there there exists a unique (up to a scalar multiple) solution $w \in A$ of the Yang-Baxter equation $[[u,w,v]] = 0$.
        \item Every $a \in A$ is invertible, and $a^{-1} \in A$.
    \end{enumerate}
    
    Then $(A, *, ^{-1})$ is an associative invertible magmoid. When completed with the units, $(A, *, ^{-1})$ forms a groupoid.
\end{conjecture}

We note that in Section 2 of \cite{BBF11}, the authors only conjecture that $(u*v)*w = u*(v*w)$ assuming that $u*v, v*w, (u*v)*w$, and $u*(v*w)$ are defined in the first place. We conjecture that if $u*v$ and $v*w$ are defined, then both $(u*v)*w$ and $u*(v*w)$ exist, and they are equal.

\bibliographystyle{alphaurl}
\bibliography{ybgroupoid.bib}

\begin{thebibliography}{BDFZJ12}

\bibitem[Bax82]{Bax82}
Rodney~J. Baxter.
\newblock {\em Exactly solved models in statistical mechanics}.
\newblock Academic Press, Inc. [Harcourt Brace Jovanovich, Publishers], London,
  1982.

\bibitem[BBF11]{BBF11}
Ben Brubaker, Daniel Bump, and Solomon Friedberg.
\newblock Schur polynomials and the {Y}ang-{B}axter equation.
\newblock {\em Comm. Math. Phys.}, 308(2):281--301, 2011.
\newblock \href {https://doi.org/10.1007/s00220-011-1345-3}
  {\path{doi:10.1007/s00220-011-1345-3}}.

\bibitem[BDFZJ12]{RPZ12}
Roger~E. Behrend, Philippe Di~Francesco, and Paul Zinn-Justin.
\newblock On the weighted enumeration of alternating sign matrices and
  descending plane partitions.
\newblock {\em J. Combin. Theory Ser. A}, 119(2):331--363, 2012.
\newblock URL:
  \url{https://doi-org.stanford.idm.oclc.org/10.1016/j.jcta.2011.09.004}, \href
  {https://doi.org/10.1016/j.jcta.2011.09.004}
  {\path{doi:10.1016/j.jcta.2011.09.004}}.

\bibitem[BS82]{BS82}
V.~V. Bazhanov and Yu.~G. Stroganov.
\newblock Trigonometric and {$S_{n}$} symmetric solutions of triangle equations
  with variables on the faces.
\newblock {\em Nuclear Phys. B}, 205(4):505--526, 1982.
\newblock URL:
  \url{https://doi-org.stanford.idm.oclc.org/10.1016/0550-3213(82)90075-X},
  \href {https://doi.org/10.1016/0550-3213(82)90075-X}
  {\path{doi:10.1016/0550-3213(82)90075-X}}.

\bibitem[Buc16]{B16}
Valentin Buciumas.
\newblock Quantum groups obtained from solutions to the parametrized
  yang-baxter equation.
\newblock {\em arXiv preprint arXiv:1602.04262}, 2016.

\bibitem[Dri87]{D87}
V.~G. Drinfel\'d.
\newblock Quantum groups.
\newblock In {\em Proceedings of the {I}nternational {C}ongress of
  {M}athematicians, {V}ol. 1, 2 ({B}erkeley, {C}alif., 1986)}, pages 798--820.
  Amer. Math. Soc., Providence, RI, 1987.

\bibitem[FW70]{FW70}
Chungpeng Fan and F~Yu Wu.
\newblock General lattice model of phase transitions.
\newblock {\em Physical Review B}, 2(3):723, 1970.

\bibitem[KBI93]{KBI93}
V.~E. Korepin, N.~M. Bogoliubov, and A.~G. Izergin.
\newblock {\em Quantum inverse scattering method and correlation functions}.
\newblock Cambridge Monographs on Mathematical Physics. Cambridge University
  Press, Cambridge, 1993.
\newblock \href {https://doi.org/10.1017/CBO9780511628832}
  {\path{doi:10.1017/CBO9780511628832}}.

\bibitem[Pau35]{Pau35}
Linus Pauling.
\newblock The structure and entropy of ice and of other crystals with some
  randomness of atomic arrangement.
\newblock {\em Journal of the American Chemical Society}, 57(12):2680--2684,
  1935.

\bibitem[Yan67]{Y67}
C.~N. Yang.
\newblock Some exact results for the many-body problem in one dimension with
  repulsive delta-function interaction.
\newblock {\em Phys. Rev. Lett.}, 19:1312--1315, 1967.
\newblock URL:
  \url{https://doi-org.stanford.idm.oclc.org/10.1103/PhysRevLett.19.1312},
  \href {https://doi.org/10.1103/PhysRevLett.19.1312}
  {\path{doi:10.1103/PhysRevLett.19.1312}}.

\end{thebibliography}

\end{document}